\documentclass[12pt]{article}    % onecolumn (standardformat)

\usepackage[margin=0.75in]{geometry} %In Folder
\usepackage{amsmath,amssymb,amsthm}
\usepackage{graphicx}  %%  ͼƬ°ü
\usepackage{subfig}    %% ×Óͼ°ü
\newtheorem{theorem}{Theorem}[section]
\newtheorem{proposition}[theorem]{Proposition}
\newtheorem{lemma}[theorem]{Lemma}
\newtheorem{corollary}[theorem]{Corollary}

\newtheorem{remark}{Remark}

\newcommand{\al}{\alpha}
\newcommand{\bt}{\beta}

\newcommand{\s}{\sigma}
\newcommand{\be}{\begin{equation}}
\newcommand{\ee}{\end{equation}}
\newcommand{\bea}{\begin{eqnarray}}
\newcommand{\eea}{\end{eqnarray}}
\newcommand{\no}{\nonumber}
\numberwithin{equation}{section}
\begin{document}

\title{\Large Asymptotic Properties of Some Freud Polynomials}
\author{Chao Min\thanks{School of Mathematical Sciences, Huaqiao University, Quanzhou 362021, China; chaomin@hqu.edu.cn}, Liwei Wang\thanks{School of Mathematical Sciences, Huaqiao University, Quanzhou 362021, China; liweiwang@stu.hqu.edu.cn}\: and Yang Chen\thanks{Department of Mathematics, University of Macau, Macau, China; yangbrookchen@yahoo.co.uk}}


\date{November 12, 2023}
\maketitle
\begin{abstract}
We study the asymptotic properties of monic orthogonal polynomials (OPs) with respect to some Freud weights when the degree of the polynomial tends to infinity, including the asymptotics of the recurrence coefficients, the nontrivial leading coefficients of the monic OPs, the associated Hankel determinants and the squares of $L^2$-norm of the monic OPs. These results are derived from the combination of the ladder operator approach, Dyson's Coulomb fluid approach and some recent results in the literature.
\end{abstract}

$\mathbf{Keywords}$: Orthogonal polynomials; Freud weights; Recurrence coefficients; Ladder operators;

Hankel determinants; Asymptotic expansions.

$\mathbf{Mathematics\:\: Subject\:\: Classification\:\: 2020}$: 33C45, 42C05, 41A60.

\section{Introduction}
In the seminal paper \cite{Magnus}, Magnus studied the relationship between recurrence coefficients of semi-classical orthogonal polynomials (OPs) and Painlev\'{e} equations. In particular, he considered the simplest nontrivial Freud weight
\be\label{wei}
w(x)=w(x;t)=\mathrm{e}^{-x^{2m}+tx^{2}},\qquad x\in \mathbb{R}
\ee
with $t\in \mathbb{R},\;m=2, 3$. This is indeed a semi-classical weight since it satisfies the Pearson equation (see e.g. \cite[Section 1.1.1]{VanAssche})
$$
\frac{d}{dx}(\s(x)w(x))=\tau(x)w(x),
$$
with $\s(x)=x,\; \tau(x)=1+2t x^2-2mx^{2m}$.
In the $m=2$ case, it was shown that the recurrence coefficient of the corresponding OPs satisfies the Painlev\'{e} IV equation and discrete Painlev\'{e} I equation simultaneously. This fact was first observed by Fokas, Its and Kitaev \cite{Fokas}; see also \cite{FIZ,Filipuk,Clarkson3}. In the $m=3$ case, it is not clear if there is any direct relation between the recurrence coefficient of the corresponding OPs and the (continuous) Painlev\'{e} equations, but it is related to the discrete Painlev\'{e} I hierarchy \cite{Clarkson2,Fokas,WZC}. For the general integer $m\geq 2$ case, see \cite{Clarkson4,Fokas2}.

Let $P_{n}(x;t),\; n=0,1,2,\ldots$, be the \textit{monic} polynomials of degree $n$ orthogonal with respect to the weight (\ref{wei}) with general integer $m\geq 2$, such that
\be\label{or}
\int_{-\infty}^{\infty}P_{j}(x;t)P_{k}(x;t)w(x;t)dx=h_{j}(t)\delta_{jk},\qquad j, k=0,1,2,\ldots,
\ee
where $h_j(t)>0$ is the square of $L^2$-norm of the OPs, and $\delta_{jk}$ is the Kronecker delta.
Note that the weight $w(x; t)$ is even, then we have \cite[p. 21]{Chihara}
$$
P_n(-x;t)=(-1)^n P_n(x;t).
$$
That means $P_n(x;t)$ contains only even (odd) powers of $x$ when $n$ is even (odd). It follows that $P_{n}(x;t)$ has the expansion
\be\label{expan}
P_{n}(x;t)=x^{n}+\mathrm{p}(n,t)x^{n-2}+\cdots,\qquad n=0,1,2,\ldots,
\ee
where $\mathrm{p}(n,t)$ is the nontrivial leading coefficient of $P_{n}(x;t)$, and $\mathrm{p}(0,t)=\mathrm{p}(1,t)=0$.

A fundamental property of OPs is that they satisfy the three-term recurrence relation. For the problem (\ref{or}), it reads \cite[p. 18-21]{Chihara}
\be\label{rr}
xP_{n}(x;t)=P_{n+1}(x;t)+\beta_{n}(t)P_{n-1}(x;t),
\ee
supplemented by the initial conditions $P_{0}(x;t)=1,\; \beta_{0}(t)P_{-1}(x;t)=0$. From (\ref{or}), (\ref{expan}) and (\ref{rr}), the recurrence coefficient $\bt_n(t)$ can be expressed in the following two forms:
\be\label{al}
\beta_{n}(t)=\mathrm{p}(n,t)-\mathrm{p}(n+1,t),
\ee
\be\label{be}
\beta_{n}(t)=\frac{h_{n}(t)}{h_{n-1}(t)}.
\ee
Taking a telescopic sum of (\ref{al}), we get an important identity
\be\label{sum}
\sum_{j=0}^{n-1}\beta_j(t)=-\mathrm{p}(n,t).
\ee
In addition, it can be seen from (\ref{al}) that $\bt_0(t)=0$.

From (\ref{or}) we have
$$
h_{n}(t)=\int_{-\infty}^{\infty}P_{n}^2(x;t)w(x;t)dx.
$$
Taking a derivative with respect to $t$ gives
$$
h_n'(t)=\int_{-\infty}^{\infty}x^2P_{n}^2(x;t)w(x;t)dx=(\bt_n(t)+\beta_{n+1}(t))h_n(t),
$$
i.e.,
\be\label{eq1}
\frac{d}{dt}\ln h_n(t)=\bt_n(t)+\bt_{n+1}(t).
\ee
The combination of (\ref{be}) and (\ref{eq1}) produces the Volterra lattice (also known as the discrete KdV equation, the Kac-van Moerbeke lattice and the Langmuir lattice) \cite[Theorem 2.4]{VanAssche}
\begin{equation}\label{eq100}
\bt_n'(t)=\bt_n(t)(\bt_{n+1}(t)-\bt_{n-1}(t)).
\end{equation}
On the other hand, from (\ref{or}) we also have
$$
\int_{-\infty}^{\infty}P_{n}(x;t)P_{n-2}(x;t)w(x;t)dx=0.
$$
By taking a derivative with respect to $t$, we obtain
$$
\frac{d}{d t} \mathrm{p}(n, t)=-\frac{1}{h_{n-2}(t)} \int_{-\infty}^{\infty}x^2P_n(x, t) P_{n-2}(x;t) w(x;t) d x=-\frac{h_n(t)}{h_{n-2}(t)}=-\bt_n(t)\bt_{n-1}(t).
$$

The Hankel determinant generated by the weight (\ref{wei}) with general integer $m\geq 2$ is defined by
$$
D_n(t):=\operatorname{det}\left(\mu_{j+k}(t)\right)_{j, k=0}^{n-1}=\left|\begin{array}{cccc}
\mu_0(t) & \mu_1(t) & \cdots & \mu_{n-1}(t) \\
\mu_1(t) & \mu_2(t) & \cdots & \mu_n(t) \\
\vdots & \vdots & & \vdots \\
\mu_{n-1}(t) & \mu_n(t) & \cdots & \mu_{2 n-2}(t)
\end{array}\right|,
$$
where $\mu_j(t)$ is the $j$th moment given by
\begin{equation*}
\mu_j(t):=\int_{-\infty}^{\infty} x^j w(x ; t) d x.
\end{equation*}
It was shown in \cite{Clarkson4} that the moments are related to the generalized hypergeometric functions.

As a well-known fact, the Hankel determinant $D_n(t)$ can be expressed in the following form of the product of $h_j(t)$ \cite[(2.1.6)]{Ismail}:
\begin{equation}\label{dnyuhj}
D_n(t)=\prod_{j=0}^{n-1} h_j(t).
\end{equation}
An easy consequence of (\ref{be}) and (\ref{dnyuhj}) shows the relation between the recurrence coefficient $\bt_n(t)$ and the Hankel determinant $D_n(t)$:
$$
\bt_n(t)=\frac{D_{n+1}(t)D_{n-1}(t)}{D_{n}^2(t)}.
$$

In random matrix theory (RMT) \cite{Mehta,Deift,Forrester}, the Hankel determinant $D_n(t)$ is equal to the partition function of the corresponding Freud unitary ensemble \cite[(2.2.11)]{Szego}, namely,
\be\label{dnt}
D_n(t)=\frac{1}{n!}\int_{(-\infty,\infty)^n}\prod_{1\leq i<j\leq n}(x_i-x_j)^2\prod_{k=1}^n \mathrm{e}^{-x_k^{2m}+tx_k^2}dx_k.
\ee
In the $t=0$ case, the large $n$ asymptotics (without higher order terms) of $D_n(0)$ in (\ref{dnt}) can be derived from a recent result of Claeys, Krasovsky and Minakov \cite{CKM}.
\begin{lemma}
As $n\rightarrow\infty$, the following asymptotics of $D_n(0)$ hold for $m\in \mathbb{Z}^{+}$,
\be\label{dn0}
\ln D_n(0)=\frac{n^2\ln n}{2m}+\left(\ln \frac{A}{2}-\frac{3}{4m}\right)n^2+n\ln (2\pi)-\frac{\ln n}{12}+c(m)+o(1),
\ee
where $A$ and $c(m)$ are constants independent of $n$, given by
$$
A=\left(\frac{\Gamma(m)\Gamma\left(\frac{1}{2}\right)}{\Gamma\left(m+\frac{1}{2}\right)}\right)^{\frac{1}{2m}}
$$
and
$$
c(m)=\zeta'(-1)-\frac{\ln m}{12}.
$$
Here $\zeta(\cdot)$ is the Riemann zeta function.
\end{lemma}
\begin{proof}
With a simple change of variables, $x_j=n^{\frac{1}{2m}}y_j,\; j=1,2,...,n$, it can be found from (\ref{dnt}) (with $t=0$) that
$$
D_n(0)=n^{\frac{n^2}{2m}}\cdot\frac{Z_n}{n!},
$$
where
$$
Z_n=\int_{(-\infty,\infty)^n}\prod_{1\leq i<j\leq n}(y_i-y_j)^2\prod_{k=1}^n \mathrm{e}^{-ny_k^{2m}}dy_k.
$$
Using Proposition 1.1 and Theorem 1.9 in \cite{CKM}, we obtain the desired result.
\end{proof}
\begin{remark}
We mention that the large $n$ asymptotics (without higher order terms) of Hankel determinants for very general Gaussian, Laguerre and Jacobi type weights have been obtained in a recent work by Charlier and Gharakhloo \cite{CG} using the Riemann-Hilbert approach.
\end{remark}

The ladder operator approach developed by Chen and Ismail \cite{Chen1997} (see also \cite{ChenIsmail2005}) is very useful to study the recurrence coefficients of OPs and the associated Hankel determinants. For the monic OPs with respect to the weight (\ref{wei}) with general integer $m\geq 2$, they satisfy a pair of ladder operator equations:
$$
\left(\frac{d}{dx}+B_{n}(x)\right)P_{n}(x)=\beta_{n}A_{n}(x)P_{n-1}(x),
$$
$$
\left(\frac{d}{dx}-B_{n}(x)-\mathrm{v}'(x)\right)P_{n-1}(x)=-A_{n-1}(x)P_{n}(x),
$$
where $\mathrm{v}(x):=-\ln w(x)$ is the potential and
\be\label{an}
A_{n}(x):=\frac{1}{h_{n}}\int_{-\infty}^{\infty}\frac{\mathrm{v}'(x)-\mathrm{v}'(y)}{x-y}P_{n}^{2}(y)w(y)dy,
\ee
\be\label{bn}
B_{n}(x):=\frac{1}{h_{n-1}}\int_{-\infty}^{\infty}\frac{\mathrm{v}'(x)-\mathrm{v}'(y)}{x-y}P_{n}(y)P_{n-1}(y)w(y)dy.
\ee
For convenience, we shall not display the $t$-dependence of many quantities most of the time.

Using the definitions (\ref{an}) and (\ref{bn}), it can be proved that $A_n(x)$ and $B_n(x)$ satisfy the compatibility conditions \cite{Chen1997,ChenIsmail2005}
\be
B_{n+1}(x)+B_{n}(x)=x A_{n}(x)-\mathrm{v}'(x), \tag{$S_{1}$}
\ee
\be
1+x(B_{n+1}(x)-B_{n}(x))=\beta_{n+1}A_{n+1}(x)-\beta_{n}A_{n-1}(x). \tag{$S_{2}$}
\ee
A suitable combination of ($S_{1}$) and ($S_{2}$) produces the sum rule \cite{ChenIts}
\be
B_{n}^{2}(x)+\mathrm{v}'(x)B_{n}(x)+\sum_{j=0}^{n-1}A_{j}(x)=\beta_{n}A_{n}(x)A_{n-1}(x), \tag{$S_{2}'$}
\ee
which usually gives us a better insight into the recurrence coefficients of the OPs compared with equation $(S_{2})$.

The rest of the paper is organized as follows. In Section 2, we apply the ladder operator approach to the monic OPs with the weight (\ref{wei}) when $m=2$, and obtain some important identities for the recurrence coefficients. Based on these results, we derive the large $n$ asymptotic expansions of the recurrence coefficient $\bt_n(t)$, the nontrivial leading coefficient $\mathrm{p}(n,t)$, the Hankel determinant $D_n(t)$ and the square of $L^2$-norm of the monic OPs, $h_n(t)$, for fixed $t\in \mathbb{R}$. In Section 3, we repeat the development of Section 2, but for the monic OPs with the weight (\ref{wei}) in the $m=3$ case. The conclusions and discussion are given in Section 4.

%In this paper, we will use the ladder operator approach to derive some important identities for the recurrence coefficients of the monic OPs with the weight (\ref{wei}), especially involving the sums of the recurrence coefficients. This will play a key role to obtain the integral representations of the Hankel determinants in terms of the recurrence coefficients. Then the large $n$ asymptotics of the Hankel determinants can be obtained from the asymptotics of the recurrence coefficients. We mainly consider the $m=2$ and $m=3$ cases. It will be seen that our method can be applied to study the higher $m$ cases.

\section{The $m=2$ Case}
In this section, we consider the $m=2$ case. The weight function reads
\be\label{wei2}
w(x;t)=\mathrm{e}^{-x^{4}+tx^{2}},\qquad x\in \mathbb{R}
\ee
with $t\in \mathbb{R}$. We have
$$
\mathrm{v}(x)=-\ln w(x)=x^4-t x^2
$$
and
\be\label{vp1}
\frac{\mathrm{v}^{\prime}(x)-\mathrm{v}^{\prime}(y)}{x-y}=4 x^2+4 x y+4 y^2-2 t.
\ee
Substituting (\ref{vp1}) into (\ref{an}) and (\ref{bn}), we obtain
\begin{equation}\label{an1}
	A_n(x)=4 x^2-2 t+4 \beta_n+4 \beta_{n+1},
\end{equation}
\begin{equation}\label{bn1}
	B_n(x)=4 x \beta_n,
\end{equation}
where use has been made of the orthogonality and the three-term recurrence relation.

With the expressions of $A_n(x)$ and $B_n(x)$ in (\ref{an1}) and (\ref{bn1}), we find that ($S_{1}$) holds automatically. From ($S_2'$) we obtain the following two equations:
\begin{equation}\label{s1}
4 \beta_n\left(\beta_{n-1}+\beta_n+\beta_{n+1}-\frac{t}{2}\right)=n,
\end{equation}
\begin{equation}\label{s2}
\sum_{j=0}^{n-1}\left(\beta_j+\beta_{j+1}\right)-\frac{nt}{2}=\beta_n\left[t^2-2 t\left(\beta_{n-1}+2 \beta_n+\beta_{n+1}\right)+4\left(\beta_{n-1}+\beta_n\right)\left(\beta_n+\beta_{n+1}\right)\right].
\end{equation}
Equation (\ref{s1}) is known as the discrete Painlev\'{e} I equation and it was derived by many authors; see e.g. \cite{BS,Magnus}, \cite[Section 2.1]{VanAssche} and the $t=0$ case in \cite{Freud,Shohat,VanAssche2}.

\begin{remark}
The combination of (\ref{eq100}) and (\ref{s1}) will show that $\bt_n$ satisfies the Painlev\'{e} IV equation; see, for example, \cite[p. 231]{Magnus}.
\end{remark}
In the following, we will study several large $n$ asymptotics of our problem for fixed $t\in \mathbb{R}$. First we show the result about the asymptotics of the recurrence coefficient.
\begin{theorem}\label{thm1}
For fixed $t\in \mathbb{R}$, the recurrence coefficient $\beta_{ n }(t)$ has the asymptotic expansion as $n\rightarrow\infty$,
\begin{equation}\label{betanzhangk}
\beta_n(t)= \frac{n^{1 / 2}}{2\sqrt{3}}+\frac{ t}{12}+\frac{t^2}{48\sqrt{3}\: n^{1/2}}-\frac{t^4-48}{2304 \sqrt{3}\: n^{3/2}}+\frac{t}{288n^{2}}+\frac{t^2\left(t^4-144\right)}{55296\sqrt{3}\:n^{5/2}}-\frac{t^3}{1728n^3}+O\big(n^{-7/2}\big).
\end{equation}
\end{theorem}
\begin{proof}
See Clarke and Shizgal \cite[p. 144]{CS} and also Clarkson and Jordaan \cite[Theorem 4.1]{Clarkson1}. The discrete Painlev\'{e} I equation (\ref{s1}) played an important role in the proof. We produce more terms by using their method here.
\end{proof}

\begin{theorem}
The nontrivial leading coefficient $\mathrm{p}(n,t)$ has the large $n$ asymptotic expansion
\begin{align}
\mathrm{p}(n, t)  =&-\frac{n^{3 / 2}}{3\sqrt{3}}-\frac{nt}{12}-\frac{\left(t^2-6\right)  n^{1 / 2}}{24\sqrt{3}}-\frac{t\left(t^2-9\right)}{216}-\frac{t^4-12 t^2-24 }{1152\sqrt{3}\: n^{1 / 2}} \no\\[8pt]
&+\frac{t}{288 n}+\frac{t^6-18 t^4-72 t^2+864 }{82944\sqrt{3}\: n^{3 / 2}}-\frac{t(t^2-6)}{3456n^2}+O\big(n^{-5/2}\big).\no
\end{align}
\end{theorem}
\begin{proof}
It follows from (\ref{s2}) and (\ref{sum}) that $\mathrm{p}(n,t)$ can be expressed in terms of $\bt_n$:
\begin{equation}\label{pn1}
\mathrm{p}(n, t)=\frac{1}{2} \beta_n+t \beta_n^2-2 \beta_n\left(\beta_{n-1}+\beta_n\right)\left(\beta_n+\beta_{n+1}\right),
\end{equation}
where we have used (\ref{s1}) to simplify the result.
Substituting (\ref{betanzhangk}) into (\ref{pn1}), we obtain the desired result by taking a large $n$ limit.
\end{proof}
\begin{remark}\label{re}
Eliminating $\bt_{n+1}$ from (\ref{s1}) and (\ref{pn1}), we find
$$
2\mathrm{p}(n, t)= \beta_n+2t \beta_n^2-(\beta_{n-1}+\beta_n)(n+2t\bt_n-4\bt_{n-1}\bt_n).
$$
By making use of (\ref{al}), we obtain the second-order difference equation satisfied by $\mathrm{p}(n,t)$:
\begin{align}
&\big(\mathrm{p}(n-1)-\mathrm{p}(n+1)\big)\left[n+2t\big(\mathrm{p}(n)-\mathrm{p}(n+1)\big)
-4\big(\mathrm{p}(n-1)-\mathrm{p}(n)\big)\big(\mathrm{p}(n)-\mathrm{p}(n+1)\big)\right]\no\\
&+\mathrm{p}(n)+\mathrm{p}(n+1)-2t\big(\mathrm{p}(n)-\mathrm{p}(n+1)\big)^2=0,\no
\end{align}
where we do not show the $t$-dependence of $\mathrm{p}(n,t)$ and $\mathrm{p}(n\pm 1,t)$ for simplicity.
\end{remark}
Before we consider the large $n$ asymptotics of the Hankel determinant $D_n(t)$, we derive the asymptotics (with higher order terms) in the special case $t=0$.
\begin{proposition}
The Hankel determinant $D_n(0)$ has the large $n$ asymptotic expansion
\begin{align}\label{dn0a}
\ln D_n(0)=&\:\frac{n^2\ln n}{4} -\frac{(3+2\ln 12)n^2}{8}+n \ln (2 \pi)-\frac{\ln n}{12} +\zeta^{\prime}(-1)-\frac{\ln 2}{12}\no\\[8pt]
&-\frac{89}{11520 n^2}+\frac{6619}{2322432 n^4}+O(n^{-6}).
\end{align}
\end{proposition}
\begin{proof}
Notice that the potential $\mathrm{v}(x)=x^4$ is convex on $\mathbb{R}$ in the $t=0$ case. This enables us to use Dyson's Coulomb fluid approach \cite{Dyson} to derive the large $n$ asymptotics of the Hankel determinant $D_n(0)$. It was shown by Chen and Ismail \cite{ChenIsmail} that $F_n:=-\ln D_n(0)$ is approximated by the free energy $F[\s]$ associated with an equilibrium density $\s(x)$ supported on a single interval for sufficiently large $n$. Following the similar procedure in \cite{MC1,MinJMP} to compute $F[\s]$ and using the fact
$$
\bt_n(0)=\frac{D_{n+1}(0)D_{n-1}(0)}{D_{n}^2(0)}
$$
to determine the precise coefficients in the expansion form of $\ln D_n(0)$, we find
\be\label{equ1}
\ln D_n(0)=\frac{n^2\ln n}{4} -\frac{(3+2\ln 12)n^2}{8}+C_1n-\frac{\ln n}{12} +C_0-\frac{89}{11520 n^2}+\frac{6619}{2322432 n^4}+O(n^{-6}),
\ee
where $C_1$ and $C_0$ are two undetermined constants.

On the other hand, from (\ref{dn0}) we have
\be\label{equ2}
\ln D_n(0)=\frac{n^2\ln n}{4} -\frac{(3+2\ln 12)n^2}{8}+n \ln (2 \pi)-\frac{\ln n}{12} +\zeta^{\prime}(-1)-\frac{\ln 2}{12}+o(1).
\ee
The combination of (\ref{equ1}) and (\ref{equ2}) gives the desired result in (\ref{dn0a}).
\end{proof}
\begin{theorem}\label{thm3}
As $n \rightarrow \infty$, the Hankel determinant $D_n(t)$ has the following asymptotic expansion:
\begin{align}\label{Dntzhang}
\ln D_n(t)  =&\:\frac{n^2\ln n}{4}-\left(\frac{3}{8}+\frac{\ln 12}{4}\right)n^2 +\frac{2  n^{3 / 2}t}{3 \sqrt{3}}+\left[\frac{t^2 }{12}+\ln (2 \pi)\right]n+\frac{  n^{1/2}t^3}{36 \sqrt{3}}-\frac{\ln n}{12}+\frac{ t^4}{432}+\zeta'(-1)\no\\[8pt]
&-\frac{\ln 2}{12} +\frac{t( t^4-120)}{2880 \sqrt{3}\:n^{1/2}}-\frac{t^2}{288 n}-\frac{t^3( t^4-168)}{290304 \sqrt{3}\:n^{3/2}}+\frac{5 t^4-267}{34560 n^2}+O\big(n^{-5/2}\big).
\end{align}
%\begin{equation}\label{hntzhang}
%\ln h_n(t)=\frac{1}{2} n \ln \frac{n}{12}-\frac{1}{2} n+\frac{t n^{1/2}}{\sqrt{3}}+\frac{1}{4} \ln n+\frac{1}{12} t^2+\ln \pi+\frac{1}{4} \ln \frac{4}{3}+o(1).
%\end{equation}
\end{theorem}
\begin{proof}
From (\ref{dnyuhj}) and (\ref{eq1}), we have
$$
\frac{d}{d t} \ln D_n(t)=\frac{d}{d t} \ln \prod_{j=0}^{n-1} h_j(t)=\sum_{j=0}^{n-1} \frac{d}{d t} \ln h_j(t)= \sum_{j=0}^{n-1}\left(\beta_j+\beta_{j+1}\right).
$$
By making using of (\ref{s2}) and (\ref{s1}), we find
$$
\frac{d}{d t} \ln D_n(t)=4 \beta_n\left(\beta_{n-1}+\beta_n\right)\left(\beta_n+\beta_{n+1}\right)-2 t \beta_n^2.
$$
It follows that
$$
\ln \frac{D_n(t)}{D_n(0)}=\int_0^t\left[4 \beta_n(s)\big(\beta_{n-1}(s)+\beta_n(s)\big)\big(\beta_n(s)+\beta_{n+1}(s)\big)-2 s \beta_n^2(s)\right] d s.
$$
Substituting (\ref{betanzhangk}) into the above and taking a large $n$ limit, we obtain
\begin{align}
\ln \frac{D_n(t)}{D_n(0)}=&\:\frac{2 n^{3 / 2}t}{3 \sqrt{3}}+\frac{ nt^2}{12}+\frac{ n^{1/2}t^3 }{36 \sqrt{3}}+\frac{ t^4}{432}+\frac{t( t^4-120)}{2880 \sqrt{3}\:n^{1/2}}-\frac{t^2}{288 n}-\frac{t^3( t^4-168)}{290304 \sqrt{3}\:n^{3/2}}\no\\[8pt]
&+\frac{t^4}{6912n^2}+O\big(n^{-5/2}\big).\no
\end{align}
By making use of (\ref{dn0a}), we establish the theorem.
%And we can use (\ref{dnyuhj}) to get
%\begin{equation}
%\ln h_n(t)=\ln D_{n+1}(t)-\ln D_n(t).
%\end{equation}
%and combine (\ref{Dntzhang}), we can get (\ref{hntzhang}).
%This completes the proof.
\end{proof}
\begin{corollary}
The quantity $h_n(t)$ has the large $n$ asymptotic expansion
\begin{align}\label{hntzhang}
\ln h_n(t)=&\:\frac{n \ln n}{2} -\frac{1}{2}(1+\ln 12)n +\frac{n^{1/2}t }{\sqrt{3}}+\frac{\ln n}{4} +\frac{t^2}{12} -\frac{\ln 12}{4} +\ln (2\pi)+\frac{ t (t^2+18)}{72 \sqrt{3}\:n^{1/2}}\no\\[8pt]
&-\frac{ t (t^4+20 t^2+120)}{5760 \sqrt{3}\:n^{3/2}}+\frac{t^2+6}{288 n^2}+O\big(n^{-5/2}\big).
\end{align}
\end{corollary}
\begin{proof}
It is easy to see from (\ref{dnyuhj}) that
$$
\ln h_n(t)=\ln D_{n+1}(t)-\ln D_n(t).
$$
Substituting (\ref{Dntzhang}) into the above, we obtain (\ref{hntzhang}) by taking a large $n$ limit.
\end{proof}
Finally, we mention that the asymptotic behavior of the polynomials orthogonal with respect to the weight (\ref{wei2}) has been studied in \cite{BI,WWX,WZ1} and applied to random matrix theory in \cite{BI,BI2}; see also \cite{BI3,WZ2} for the more general case.

\section{The $m=3$ Case}
In this section, we study the $m=3$ case, and the weight function now is
$$
w(x;t)=\mathrm{e}^{-x^{6}+t x^2},\qquad x\in \mathbb{R}
$$
with $t\in \mathbb{R}$. We have
$$
\mathrm{v}(x)=-\ln w(x)=x^6-tx^2.
$$
It follows that
\be\label{vp3}
\frac{\mathrm{v}'(x)-\mathrm{v}'(y)}{x-y}=6\left(x^4+x^3 y+x^2 y^2+x y^3+y^4\right)-2 t.
\ee
Plugging (\ref{vp3}) into the definitions of $A_n(x)$ and $B_n(x)$ in (\ref{an}) and (\ref{bn}), we obtain
\begin{equation}\label{an3}
A_n(x)=6 x^4+6x^2\left(\beta_n+\beta_{n+1}\right) -2 t+R_n(t),
\end{equation}
\begin{equation}\label{bn3}
B_n(x)=6x^3 \beta_n +x\:r_n(t),
\end{equation}
where $R_{n}(t)$ and  $r_{n}(t)$ are the auxiliary quantities given by
$$
R_n(t):=\frac{6}{h_n} \int_{-\infty}^{\infty} y^4 P_{n}^{2}(y) w(y) d y,
$$
$$
r_n(t):=\frac{6}{h_{n-1}} \int_{-\infty}^{\infty} y^3 P_n(y) P_{n-1}(y) w(y) d y.
$$

Substituting (\ref{an3}) and (\ref{bn3}) into ($S_{1}$), we find
\begin{equation}\label{m4}
R_n(t)=r_n(t)+r_{n+1}(t).
\end{equation}
Similarly, substituting (\ref{an3}) and (\ref{bn3}) into ($S_{2}'$), we obtain the following four equalities:
\begin{equation}\label{s5}
6 \beta_n^2+ r_n(t)=6 \beta_n\left(\beta_{n-1}+2 \beta_n+\beta_{n+1}\right),
\end{equation}
\begin{equation}\label{s6}
2 \beta_n r_n(t)-2 t \beta_n+ n=\beta_n\left[R_{n-1}(t)+R_n(t)-4 t+6\left(\beta_{n-1}+\beta_n\right)\left(\beta_n+\beta_{n+1}\right)\right],
\end{equation}
\begin{align}\label{s7}
&r_n^2(t)-2 t\: r_n(t)+6 \sum_{j=0}^{n-1}\left(\beta_j+\beta_{j+1}\right)\no\\
&= 6\beta_n\left[\left(\beta_n+\beta_{n+1}\right) \left(R_{n-1}(t)-2t\right)+\left(\beta_{n-1}+\beta_n\right) \left(R_n(t)-2t\right)\right],
\end{align}
$$
\sum_{j=0}^{n-1} R_j(t)-2 n t=\beta_n\left(R_n(t)-2t\right)\left(R_{n-1}(t)-2t\right).
$$
From (\ref{s5}) and (\ref{m4}), we can express $r_n(t)$ and $R_n(t)$ in terms of the recurrence coefficient $\bt_n$:
\begin{equation}\label{rn3}
r_n(t)=6\beta_n\left(\beta_{n-1}+\beta_n+\beta_{n+1}\right),
\end{equation}
\begin{equation}\label{Rn3}
R_n(t)=6\beta_n(\beta_{n-1}+\beta_n+\beta_{n+1})+6\beta_{n+1}(\beta_n+\beta_{n+1}+\beta_{n+2}).
\end{equation}
Substituting (\ref{rn3}) and (\ref{Rn3}) into (\ref{s6}), we find the fourth-order nonlinear difference equation satisfied by $\bt_n$,
\begin{align}\label{equa1}
&\:6 \beta_n\left(\beta_{n-2} \beta_{n-1}+\beta_{n-1}^2+2\beta_{n-1}\beta_{n}+\beta_{n-1}\beta_{n+1}+\beta_n^2+2\bt_n\bt_{n+1}+\bt_{n+1}^2+\beta_{n+1} \beta_{n+2}\right)\no\\
&-2t\beta_n=n.
\end{align}
This equation is within the discrete Painlev\'{e} I hierarchy \cite{CJ} and was derived by Magnus \cite{Magnus} and Wang et al. \cite{WZC}. See also Clarkson and Jordaan \cite{Clarkson2} and the $t=0$ case by Freud \cite{Freud} and Van Assche \cite{VanAssche2}.
%and using (\ref{s7}), (\ref{al}), (\ref{rn3}) and (\ref{Rn3}) we get the expression of $p(n,t)$
%\begin{equation}\label{pn2}
%\begin{aligned}
%p(n)= & \frac{1}{2} \beta_n\left\{1-6 \beta_n^3-12 \beta_n^2\left(\beta_{n-1}+\beta_{n+1}\right)-6 \beta_{n-1} %\beta_{n+1}\left(\beta_{n-2}+\beta_{n-1}+\beta_{n+1}+\beta_{n+2}\right)\right. \\
%& +2 \beta_n[t-3(\beta_{n-2} \beta_{n-1}+\left(\beta_{n-1}+\beta_{n+1}\right)^2+\beta_{n+1} \beta_{n+2})]\}
%\end{aligned}
%\end{equation}

Similarly as in the previous section, we consider the large $n$ asymptotics of the recurrence coefficient $\bt_n(t)$, the nontrivial leading coefficient $\mathrm{p}(n,t)$, the Hankel determinant $D_n(t)$ and the quantity $h_n(t)$ for fixed $t\in \mathbb{R}$ in the following discussions.
\begin{theorem}\label{thm11}
As $n\rightarrow\infty$, the recurrence coefficient $\beta_{ n }(t)$ has the asymptotic expansion for fixed $t\in \mathbb{R}$,
\begin{align}\label{betanzhangk2}
%\beta_n=\frac{n^{1/3}}{\sqrt[3]{60}}\left[1+\frac{\left(\frac{2}{15}\right)^{1 / 3} t}{3 n^{2 / 3}}+\frac{135-4 t^3}{2430 n^2}+\frac{\left(\frac{2}{15}\right)^{1 / 3} t\left(-45 + 2 t^3\right)}{3645 n^{8 / 3}}-\frac{2\left(\frac{2}{15}\right)^{2 / 3} t^2}{81 n^{10 / 3}}+o\left(n^{-11/3}\right)\right].
\beta_n=&\:\frac{n^{1/3}}{\sqrt[3]{60}}+\frac{t}{3\sqrt[3]{450}\: n^{1/3}}-\frac{4 t^3-135}{2430\sqrt[3]{60}\: n^{5/3}}+\frac{t \left(2 t^3-45\right)}{3645\sqrt[3]{450}\:   n^{7/3}}-\frac{2 t^2}{1215 n^3}-\frac{16 t^6-3780 t^3+40095}{1476225 \sqrt[3]{60}\: n^{11/3}}\no\\[8pt]
&+\frac{t (16 t^6-180 t^3-112995)}{3542940 \sqrt[3]{450}\: n^{13/3}}+O(n^{-5}).
\end{align}
%\begin{equation}\label{pn2zhang}
%\begin{aligned}
%p(n,t)&= -\frac{3^{2 / 3} n^{4 / 3}}{4\times2^{2 / 3} \times5^{1 / 3}}-\frac{t n^{2 / 3}}{2\times2^{1 / 3} \times 15^{2 / 3}}+\frac{n^{1 / 3}}{2 \times 2^{2 / 3} \times 15^{1 / 3}}-\frac{t^2}{90}+\frac{t}{6 \times 2^{1 / 3} \times 15^{2 / 3}n^{1 / 3}}+ \\
%&+ \frac{45-2 t^3}{810 \times 2^{2 / 3} \times 15^{1 / 3}n^{2 / 3}}+\frac{t^4}{2430 \times 2^{1 / 3} \times 15^{2/ 3}n^{4 / 3}}+\frac{135-4 t^3}{4860 \times 2^{2 / 3} \times 15^{1 / 3}n^{5/ 3}}+o(n^{-2}). \\
%\end{aligned}
%\end{equation}
\end{theorem}
\begin{proof}
See Clarkson and Jordaan \cite[Theorem 4.9]{Clarkson2}. We produce more terms by using their method here.
\end{proof}
\begin{theorem}
The nontrivial leading coefficient $\mathrm{p}(n,t)$ has the large $n$ asymptotic expansion
\begin{align}
\mathrm{p}(n,t)=& -\frac{3\: n^{4 / 3}}{4\sqrt[3]{60}}-\frac{n^{2 / 3}t }{2\sqrt[3]{450}}+\frac{n^{1 / 3}}{2 \sqrt[3]{60}}-\frac{t^2}{90}+\frac{t}{6 \sqrt[3]{450}\:n^{1 / 3}}- \frac{2 t^3-45}{810 \sqrt[3]{60}\:n^{2 / 3}}\no\\[8pt]
&+\frac{t^4}{2430 \sqrt[3]{450}\:n^{4 / 3}}-\frac{4 t^3-135}{4860\sqrt[3]{60}\: n^{5/ 3}}+O(n^{-2}).\no
\end{align}
\end{theorem}
\begin{proof}
From (\ref{s7}) and (\ref{sum}), we can express $\mathrm{p}(n,t)$ in terms of the recurrence coefficient $\bt_n$:
\begin{align}\label{equa2}
\mathrm{p}(n,t)=&\:\frac{1}{2}\bt_n\big[1+2t\bt_n-6\bt_n^2(\beta_{n-1}+\beta_n+\beta_{n+1})-6(\beta_n+\beta_{n+1})\beta_{n-1}(\beta_{n-2}+\beta_{n-1}+\beta_{n})\no\\[8pt]
&-6(\beta_{n-1}+\beta_{n})\beta_{n+1}(\beta_{n}+\beta_{n+1}+\beta_{n+2})\big],
\end{align}
where we have used (\ref{rn3}) and (\ref{Rn3}) to eliminate $r_n(t)$ and $R_n(t)$. Substituting (\ref{betanzhangk2}) into the above and letting $n\rightarrow\infty$, we obtain the result in the theorem.
\end{proof}
\begin{remark}
Similarly as in the previous section (see Remark \ref{re}), one can derive a fourth-order difference equation satisfied by $\mathrm{p}(n,t)$ from (\ref{al}), (\ref{equa1}) and (\ref{equa2}). We leave it to the interested reader.
\end{remark}
\begin{proposition}
The Hankel determinant $D_n(0)$ has the large $n$ asymptotic expansion
\begin{align}\label{dn0b}
\ln D_n(0)=&\:\frac{n^2\ln n}{6} -\left(\frac{1}{4}+\frac{\ln 60}{6}\right) n^2+n \ln (2 \pi)-\frac{\ln n}{12} +\zeta^{\prime}(-1)-\frac{\ln 3}{12}\no\\[8pt]
&-\frac{7}{648 n^2}+\frac{8521}{1837080 n^4}+O(n^{-6}).
\end{align}
\end{proposition}
\begin{proof}
In this $t=0$ case, the potential $\mathrm{v}(x)=x^6$ is convex on $\mathbb{R}$. By Dyson's Coulomb fluid approach and following the similar procedure in \cite{MC1,MinJMP}, we find
\be\label{equ11}
\ln D_n(0)=\frac{n^2\ln n}{6} -\left(\frac{1}{4}+\frac{\ln 60}{6}\right) n^2+\widetilde{C}_1n -\frac{\ln n}{12} +\widetilde{C}_0-\frac{7}{648 n^2}+\frac{8521}{1837080 n^4}+O(n^{-6}).
\ee
where $\widetilde{C}_1$ and $\widetilde{C}_0$ are two undetermined constants.

From (\ref{dn0}) we have
\be\label{equ22}
\ln D_n(0)=\frac{n^2\ln n}{6} -\left(\frac{1}{4}+\frac{\ln 60}{6}\right) n^2+n \ln (2 \pi)-\frac{\ln n}{12} +\zeta^{\prime}(-1)-\frac{\ln 3}{12}+o(1).
\ee
Combining (\ref{equ11}) and (\ref{equ22}) gives the desired result.
\end{proof}
\begin{theorem}\label{thm5}
The Hankel determinant $D_n(t)$ has the following asymptotic expansion as $n \rightarrow \infty$,
\begin{align}\label{Dnzhangk2}
\ln D_n(t)=&\:\frac{n^2\ln n}{6} -\left(\frac{1}{4}+\frac{\ln60}{6}\right)n^2+\frac{3\:  n^{4/3}t}{2 \sqrt[3]{60}}+n \ln (2 \pi)+\frac{ n^{2/3}t^2}{2 \sqrt[3]{450}}-\frac{\ln n}{12} +\frac{t^3}{135}+\zeta^{\prime}(-1)-\frac{\ln 3}{12}\no\\[8pt]
&+\frac{t(t^3-90)}{810 \sqrt[3]{60}\: n^{2/3}}
-\frac{t^5}{6075 \sqrt[3]{450}\: n^{4/3}}+\frac{16 t^3-315}{29160 n^2}+O\big(n^{-8/3}\big).
\end{align}
\end{theorem}
\begin{proof}
From (\ref{dnyuhj}) and (\ref{eq1}), we have
\be\label{e1}
\frac{d}{d t} \ln D_n(t)=\sum_{j=0}^{n-1}\left(\beta_j+\beta_{j+1}\right).
\ee
Substituting (\ref{rn3}) and (\ref{Rn3}) into (\ref{s7}) gives
\bea\label{e2}
\sum_{j=0}^{n-1}\left(\beta_j+\beta_{j+1}\right)&=&2\bt_n\big[3\bt_n^2(\beta_{n-1}+\beta_n+\beta_{n+1})+3(\beta_n+\beta_{n+1})\beta_{n-1}(\beta_{n-2}+\beta_{n-1}+\beta_{n})\no\\
&&+3(\beta_{n-1}+\beta_{n})\beta_{n+1}(\beta_{n}+\beta_{n+1}+\beta_{n+2})-t\bt_n\big].
\eea
It follows from the combination of (\ref{e1}) and (\ref{e2}) that
\bea
\ln \frac{D_n(t)}{D_n(0)}&=&2\int_0^t\bt_n(s)\big[3\bt_n^2(s)\left(\beta_{n-1}(s)+\beta_n(s)+\beta_{n+1}(s)\right)\no\\
&&+3\left(\beta_n(s)+\beta_{n+1}(s)\right)\beta_{n-1}(s)\left(\beta_{n-2}(s)+\beta_{n-1}(s)+\beta_{n}(s)\right)\no\\
&&+3\left(\beta_{n-1}(s)+\beta_{n}(s)\right)\beta_{n+1}(s)\left(\beta_{n}(s)+\beta_{n+1}(s)+\beta_{n+2}(s)\right)-s\bt_n(s)\big]ds.\no
\eea
Substituting (\ref{betanzhangk2}) into the above and taking a large $n$ limit, we obtain
\be\label{ra2}
\ln \frac{D_n(t)}{D_n(0)}
=\frac{3\:  n^{4/3}t}{2 \sqrt[3]{60}}+\frac{ n^{2/3}t^2}{2 \sqrt[3]{450}}+\frac{t^3}{135}+\frac{t(t^3-90)}{810 \sqrt[3]{60}\: n^{2/3}}
-\frac{t^5}{6075 \sqrt[3]{450}\: n^{4/3}}+\frac{2 t^3}{3645 n^2}+O\big(n^{-8/3}\big).
\ee
The combination of (\ref{ra2}) and (\ref{dn0b}) gives the result (\ref{Dnzhangk2}).
%And we can use (\ref{dnyuhj}) to get
%\begin{equation}
%\ln h_n(t)=\ln D_{n+1}(t)-\ln D_n(t).
%\end{equation}
%and combine (\ref{Dnzhangk2}), we can get (\ref{hnzhangk2}).
%This completes the proof.
\end{proof}
\begin{corollary}
The quantity $h_n(t)$ has the large $n$ asymptotic expansion
\begin{align}
\ln h_n(t)=&\:\frac{n \ln n}{3} -\frac{1}{3}(1+\ln 60) n+\frac{2\:n^{1/3}t}{\sqrt[3]{60}}+\frac{\ln n}{6}  +\ln(2\pi)-\frac{\ln 60}{6}+\frac{t^2}{3\sqrt[3]{450}\:n^{1/3}}\no\\[8pt]
&+\frac{t}{3\sqrt[3]{60}\:n^{2/3}}-\frac{1}{36n}-\frac{t^2}{18\sqrt[3]{450}\:n^{4/3}}-\frac{t^4}{1215\sqrt[3]{60}\:n^{5/3}}+O(n^{-2}).\no
\end{align}
\end{corollary}
\begin{proof}
Using the fact
$$
\ln h_n(t)=\ln D_{n+1}(t)-\ln D_n(t),
$$
and substituting (\ref{Dnzhangk2}) into the above, we obtain the desired result by taking a large $n$ limit.
\end{proof}

\section{Conclusions and Discussion}
In this paper, we have studied various large $n$ asymptotic expansions about the OPs with respect to the Freud weights (\ref{wei}). These are the asymptotics of the recurrence coefficients, the nontrivial leading coefficients of the monic OPs, the associated Hankel determinants and the squares of $L^2$-norm of the monic OPs. We mainly consider the $m=2$ and $m=3$ cases. It should be seen that our method can be applied to study the higher $m$ cases. Furthermore, we point out that we can compute the higher order terms of $n$ up to any degree for all the asymptotic expansions obtained in this paper.

% Acknowledgements
\section*{Acknowledgments}
%This work was partially supported by the National Natural Science Foundation of China under grant number 12001212, by the Fundamental Research Funds for the Central Universities under grant number ZQN-902 and by the Scientific Research Funds of Huaqiao University under grant number 17BS402.
The work of Chao Min was partially supported by the National Natural Science Foundation of China under grant number 12001212, by the Fundamental Research Funds for the Central Universities under grant number ZQN-902 and by the Scientific Research Funds of Huaqiao University under grant number 17BS402. The work of
Yang Chen was partially supported by the Macau Science and Technology Development Fund under grant number FDCT 0079/2020/A2.

\section*{Conflict of Interest}
The authors have no competing interests to declare that are relevant to the content of this article.
%This work does not have any conflict of interest.
\section*{Data Availability Statements}
Data sharing not applicable to this article as no datasets were generated or analysed during the current study.

%\section*{Conflicts of Interest}
%The authors have no conflicts of interest to declare that are relevant to the content of this article.

%\section*{Data Availability Statements}
%Data sharing not applicable to this article as no datasets were generated or analysed during the current study.

\end{document}